\documentclass{imsart-arxiv}

\RequirePackage[OT1]{fontenc}
\RequirePackage{amsthm,amsmath,amssymb}
\RequirePackage[numbers]{natbib}
\RequirePackage[colorlinks,citecolor=blue,urlcolor=blue]{hyperref}

\usepackage{graphicx}

\usepackage{here}

\usepackage{latexsym}

\usepackage{amsbsy}

\usepackage{amsthm, amsfonts}

\usepackage{tikz}
\usetikzlibrary {shapes}
\usetikzlibrary {arrows}
\usetikzlibrary {positioning}
\usetikzlibrary {calc}
\usetikzlibrary{fit}					
\usetikzlibrary{backgrounds}	

\usepackage{color}
\newcommand{\nodeset}{\mathcal{N}}

\newcommand{\dse}{\,\mbox{$\perp$}\,}

\newcommand{\cip}{\mbox{\,$\perp\!\!\!\perp$\,}}
\newcommand{\sk}{\mathrm{sk}}

\newcommand{\cd}{\,|\,}

\newcommand{\ci}{\mbox{\protect $\: \perp \hspace{-2.3ex}
\perp$ }}

\newcommand{\notdse}{\nolinebreak{\not\hspace{-1.5mm}\dse}}

\newcommand{\notci}{\nolinebreak{\not\hspace{-1.5mm}\ci}}

\newcommand{\n}[0]{\hspace*{.35em}}

\newcommand{\nn}[0]{\hspace*{.7em}}

\newcommand{\node}{\mbox {\LARGE
{$\mbox{$\circ$}$}}}

\newcommand{\fla}{\mbox{$\hspace{.05em} \prec
\!\!\!\!\!\frac{\nn \nn}{\nn}$}}

\newcommand{\fra}{\mbox{$\hspace{.05em} \frac{\nn
\nn}{\nn
}\!\!\!\!\! \succ \! \hspace{.25ex}$}}

\newcommand{\arc}{\mbox{$\hspace{.06em} \prec
\!\!\!\!\!\frac{\nn \nn}{\nn}
\!\!\!\!\!
\succ\! \hspace{.25ex}$}}

\newtheorem{prop}{Proposition}
\newtheorem{coro}{Corollary}

\newtheorem{lemma}{Lemma}

\newtheorem{alg}{Algorithm}
\newtheorem{theorem}{Theorem}
\newtheorem{example}{Example}

\startlocaldefs
\numberwithin{equation}{section}
\theoremstyle{plain}

\theoremstyle{remark}

\theoremstyle{definition}

\endlocaldefs
\begin{document}

\begin{frontmatter}
\title{On Finite Exchangeability and Conditional Independence}
\runtitle{On Finite Exchangeability and Conditional Independence}

\begin{aug}
\author{\snm{Kayvan Sadeghi}\ead[label=e1]{k.sadeghi@ucl.ac.uk}}

\address{Department of Statistical Science, University College London, Gower Street, London, WC1E 6BT, United Kingdom}


\runauthor{K. Sadeghi}

\affiliation{University College London}

\end{aug}

\begin{abstract}
We study the independence structure of finitely exchangeable distributions over random vectors and random networks. In particular, we provide necessary and sufficient conditions for an exchangeable vector so that its elements are completely independent or completely dependent. We also provide a sufficient condition for an exchangeable vector so that its elements are marginally independent. We then generalize these results and conditions for exchangeable random networks. In this case, it is demonstrated that the situation is more complex. We show that the independence structure of exchangeable random networks lies in one of six regimes that are two-fold dual to one another, represented by undirected and bidirected independence graphs in graphical model sense with graphs that are complement of each other. In addition, under certain additional assumptions, we provide necessary and sufficient conditions for the exchangeable network distributions to be faithful to each of these graphs.
\end{abstract}


\begin{keyword}
\kwd{conditional independence}
\kwd{exchangeability}
\kwd{faithfulness}
\kwd{random networks}
\end{keyword}

\end{frontmatter}

\section{Introduction}
The concept of exchangeability has been a natural and convenient assumption to impose in probability theory and for simplifying statistical models. As defined originally for random sequences, it states that any order of a finite number of samples is equally likely. The concept was later generalized for binary random arrays \citep{ald85}, and consequently for random networks in statistical network analysis. In this context, exchangeability is translated into invariance under relabeling
of the nodes of the network, whereby isomorphic graphs have the
same probabilities; see, e.g., \cite{lau17}.

Exchangeability is closely related to the concept of independent and identically distributed random variables. It is an immediate consequence of the definition that independent and identically distributed random variables are exchangeable, but the converse is not true. For infinite sequences, the converse is established by the well-known deFinetti's Theorem \citep{def31}, which implies that in any infinite sequence of exchangeable random variables, the random variables are conditionally independent and identically-distributed given the underlying distributional form. Other versions of deFinetti's Theorem exist for the generalized definitions of exchangeability \citep{sil76,dia08}.

However, for finitely exchangeable random sequences (vectors) and arrays (matrices), the converse does not hold, and the available results basically provide approximations of the infinite case; see, e.g., \cite{dia80,mat95,lau19}. In this paper, we utilize a completely different approach to study the relationship between finite exchangeability and (conditional) independence.  We employ the theory of graphical models (see e.g.\ \cite{lau96}) in order to provide the independence structure of exchangeable distributions.

In particular, we exploit the necessary and sufficient conditions, provided in \cite{sad17}, for faithfulness of probability distributions and graphs, which determine when the conditional independence structure of the distribution is exactly the same as that of a graph in graphical model sense. Thus, we, in practice, work on the induced independence model of an exchangeable probability distribution rather than the distribution itself. The specialization to finitely exchangeable random vectors leads to  necessary and sufficient conditions for an exchangeable vector to be completely independent or completely dependent, meaning that two elements of the vector are conditionally independent or dependent, respectively, given any subset of the remaining elements of the vector. These conditions are namely intersection and composition properties \cite{pea88,stu05}. We also use the results to provide a sufficient condition for an exchangeable vector to be marginally independent.

For random networks, we follow a similar procedure. It was shown in \cite{lau17} that if a distribution over an exchangeable random network could be faithful to a graph then the skeleton of the graph is only one of the four possible types: the empty graph, the complete graph, and the, so-called incidence graph, and its complement (see Proposition \ref{prop:7}). A question is then which graphs that emerge from these skeleta are faithful to the exchangeable distribution. We show that, other than complete independence or dependence, there are four other independence structures that can arise for exchangeable networks. These are faithful to a pair of dual graphs with incidence graph skeleton, and  a pair of dual graphs with the complement of the incidence graph skeleton. We then provide assumptions under which intersection and composition properties are necessary and sufficient for the exchangeable random networks to be faithful to each of these six cases.

The structure of the paper is as follows: In the next section, we provide some definitions and known results needed for this paper from graph theory, random networks, and graphical models. In Section \ref{sec:3}, we provide the results for exchangeable random vectors.  In Section \ref{sec:4}, we provide the results for exchangeable random networks by starting with providing the aforementioned possible cases in Theorem \ref{thm:2}, and then introducing them in separate subsections. We end with a short discussion on these results in Section \ref{sec:5}. We will provide technical definitions and results needed for the proof of Theorem \ref{thm:2}, and its proof, in Appendix \ref{app}.
\section{Definitions and preliminary results}\label{sec:2}
\subsection{Graph-theoretic concepts}
A \emph{(labeled) graph} is an ordered pair $G = (V,E)$ consisting of a
\emph{vertex} set $V$, which is non-empty and finite, an \emph{edge} set $E$, and a relation that with
each edge associates two vertices, called
its \emph{endpoints}. We omit the term labeled in this paper since the context does not give rise to ambiguity.
 When vertices $u$ and $v$ are the endpoints of an
edge, these are
\emph{adjacent} and we write $u\sim v$; we denote the corresponding edge as
$uv$.

In this paper, we will restrict our attention to \emph{simple} graphs, i.e.\ graphs without
loops (this assumption means that the
endpoints of each edge are distinct) or multiple edges (each pair of vertices are the endpoints
of at most one edge). Further, three \emph{types} of edges, denoted by
\emph{arrows}, \emph{arcs} (solid lines with two-headed arrows) and \emph{lines}
(solid lines without arrowheads) have been used in the literature of graphical models.  Arrows can be represented by ordered pairs of vertices, while arcs
and lines by $2$-subsets of the vertex set. However, for our purpose, except for Section \ref{sec:41}, we will distinguish only lines and arcs, which  respectively form
\emph{undirected} and \emph{bidirected} graphs.

The graphs $F=(V_F,E_F)$ and $G=(V_G,E_G)$ are considered equal if and only
if $(V_F,E_F)=(V_G,E_G)$.

A \emph{subgraph} of a graph $G = (V_G,E_G)$ is graph $F = (V_F,E_F)$ such that
$V_F\subseteq V_G$ and $E_F\subseteq E_G$ and the assignment of endpoints to
edges in $F$
is the same as in $G$.


The \emph{line graph} $L(G)$ of a graph $G=(V,E)$ is the intersection graph of
the edge set $E$, i.e.\ its vertex set is $E$ and $e_1\sim e_2$ if and only if
$e_1$ and $e_2$ have a common endpoint \citep[p.\ 168]{wes01}. We will in
particular be interested in the line graph of a complete graph, which we will
refer to as the \emph{incidence graph}. Figure~\ref{fig:2} displays the incidence graph for $V = \{1,2,3,4\}$.
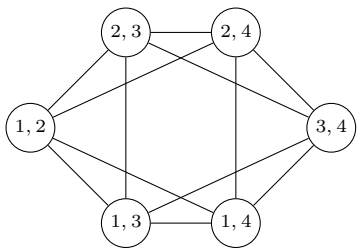
\begin{figure}[htb]
\centering
\begin{tikzpicture}[node distance = 8mm and 8mm, minimum width = 6mm]
    \begin{scope}
      \tikzstyle{every node} = [shape = circle,
      font = \scriptsize,
      minimum height = 6mm,
      inner sep = 2pt,
      draw = black,
      fill = white,
      anchor = center],
      text centered]
      \node(12) at (0,0) {$1,2$};
      \node(13) [below right = of 12] {$1,3$};
      \node(23) [above right = of 12] {$2,3$};
			\node(24) [right = of 23] {$2,4$};
      \node(14) [right = of 13] {$1,4$};
      \node(34) [below right = of 24] {$3,4$};
			
    \end{scope}

    \begin{scope}
      \draw (12) -- (13);
			\draw (12) -- (14);
			\draw (12) -- (23);
			\draw (12) -- (24);
      \draw (13) -- (23);
			\draw (13) -- (14);
			\draw (13) -- (34);
			\draw (23) -- (24);
			\draw (23) -- (34);
			\draw (14) -- (24);
			\draw (24) -- (34);\draw (14) -- (34);

    \end{scope}

%
		

    \end{tikzpicture}
		\caption{\small{The incidence graph for $V=
		    \{1,2,3,4\}$.}}
     \label{fig:2}
		\end{figure}
We denote by $L_{-}(n)$ and $L_{\leftrightarrow}(n)$, the undirected incidence graph and the bidirected incidence graph for $n$ nodes, respectively; and by  $L^c_{-}(n)$ and $L^c_{\leftrightarrow}(n)$, the undirected complement of the incidence graph and the bidirected complement of the incidence graph for $n$ nodes, respectively, where the \emph{complement} of graph $G$ refers to a graph with the same node set as $G$, but with the edge set that is the complement set of the edge set of $G$.

The \emph{skeleton} of a graph is the undirected graph where all arrowheads are removed from the graphs, i.e., all edges are replaced by lines. We denote the skeleton of a graph $G$ by $\sk(G)$.

A \emph{walk} $\omega$ is a list $\omega=\langle i_0,e_1,i_1,\dots,e_n,i_n\rangle$ of vertices and edges such that for $1\leq m\leq n$, the edge $e_m$ has endpoints $i_{m-1}$ and $i_m$. When indicating a specific walk, we may skip the edges, and only write the nodes, when the walk we are considering is clear from the context. A \emph{path} is a walk with no repeated nodes.
\subsection{Random networks}
Given a finite node set $\nodeset$ ---
representing \emph{individuals} or \emph{actors} in a given population of
interest ---  we define a
\emph{random network} over $\nodeset$ to be a collection $X = (X_d, d\in
\mathcal{D}(\nodeset) )$
of binary random variables taking values $0$ and $1$ indexed by a set
$\mathcal{D}(\nodeset)$, which is a collection of unordered pairs $ij$ of nodes in $\nodeset$. The binary random variables $X_d$ are called \emph{dyads}, and
nodes $i$ and $j$ are said to have
a \emph{tie}  if the random variable $X_{ij}$ takes the value $1$, and no tie
otherwise. 
Thus, a random network is a random
variable taking value in $\{0,1\}^{\nodeset \choose 2}$ and can, therefore,
be seen as a random {\it simple, undirected} graph with node set $\nodeset$,
whereby the ties form the random edges of the graphs.

We use the terms
network, node, and tie rather than graph, vertex, and edge to differentiate
from the terminology used in the graphical model sense. Indeed, as we shall
discuss graphical models for networks, we will also consider each dyad $d$ as a
vertex in a graph $G=(\mathcal{D},E)$ representing the dependence structure of the random
variables associated with the dyads, with the edge set of such graph
representing Markov properties of the distribution of $X$.
\subsection{Probabilistic independence models and their properties}\label{sec:propind}
An \emph{independence model} $\mathcal{J}$ over a finite set $V$ is a set of triples $\langle X,Y\cd Z\rangle$ (called \emph{independence statements}), where $X$, $Y$, and $Z$ are disjoint subsets of $V$; $Z$
may be empty, but $\langle \varnothing,Y\cd Z\rangle$ and $\langle X,\varnothing\cd Z\rangle$ are always included in $\mathcal{J}$. The independence statement $\langle X,Y\cd Z\rangle$ is read as ``$X$ is independent of $Y$ given $Z$''. Independence models may in  general have a  probabilistic interpretation, but not necessarily. Similarly, not all independence models can be easily represented by graphs. For further discussion on general independence models, see \cite{stu05}.

In order to define probabilistic independence models, consider a set $V$ and a collection of random variables
$\{X_\alpha\}_{\alpha\in V}$ with state spaces $\mathcal{X}_\alpha, \alpha\in V$ and joint distribution $P$. We let $X_A=\{X_v\}_{v\in A}$ etc.\ for each subset $A$ of $V$. For disjoint subsets $A$, $B$, and $C$ of $V$
we use a short notation $A\cip B\cd C$ to denote that $X_A$ is \emph{conditionally independent of $X_B$ given $X_C$} \cite{daw79,lau96}, i.e.\ that for any measurable $\Omega\subseteq \mathcal{X}_A$ and $P$-almost all $x_B$ and $x_C$,
$$P(X_A \in \Omega\cd X_B=x_B, X_C=x_C)=P(X_A \in \Omega\cd X_C=x_C).$$
We can now induce an independence model $\mathcal{J}(P)$ by letting
\begin{displaymath}
\langle A,B\cd C\rangle\in \mathcal{J}(P) \text{ if and only if } A\ci B\cd C \text{ w.r.t.\ $P$}.
\end{displaymath}
Similarly we use the notation $A\notci B\cd C$ for $\langle A,B\cd C\rangle\notin \mathcal{J}(P)$.

We say that non-empty $A$ and $B$ are \emph{completely independent} if, for every $C\subseteq V\setminus (A\cup B)$, $A\ci B\cd C$. Similarly, we say that $A$ and $B$ are \emph{completely dependent} if, for any $C\subseteq V\setminus (A\cup B)$, $A\notci B\cd C$.

If $A$, $B$, or $C$ has only one member $\{u\}$, $\{v\}$, or $\{w\}$, for better readability, we write $u\ci v\cd w$. We also write $A\ci B$ when $C=\varnothing$, which denotes the \emph{marginal independence} of $A$ and $B$.

A probabilistic independence model $\mathcal{J}(P)$ over a set $V$ is always a \emph{semi-graphoid} \cite{pea88}, i.e., it satisfies the four following properties for disjoint subsets $A$, $B$, $C$, and $D$ of $V$:
 \begin{enumerate}
    \item $A\ci B\cd C$ if and only if $B\ci A\cd C$ (\emph{symmetry});
    \item if $A\ci B\cup D\cd C$ then $A\ci B\cd C$ and $A\ci D\cd C$ (\emph{decomposition});
    \item if $A\ci B\cup D\cd C$ then $A\ci B\cd C\cup D$ and $A\ci D\cd C\cup B$ (\emph{weak union});
    \item if $A\ci B\cd C\cup D$ and $A\ci D\cd C$ then $A\ci B\cup D\cd C$ (\emph{contraction}).
 \end{enumerate}
Notice that the reverse implication of contraction clearly holds by decomposition and weak union. A semi-graphoid for which the reverse implication of the weak union property holds is said to be a \emph{graphoid}; that is, it also satisfies
\begin{itemize}
	\item[5.] if $A\ci B\cd C\cup D$ and $A\ci D\cd C\cup B$ then $A\ci B\cup D\cd C$ (\emph{intersection}).
\end{itemize}
Furthermore, a graphoid or semi-graphoid for which the reverse implication of the decomposition property holds is said to be \emph{compositional}, that is, it also satisfies
\begin{itemize}
	\item[6.] if $A\ci B\cd C$ and $A\ci D\cd C$ then $A\ci B\cup D\cd C$ (\emph{composition}).
\end{itemize}
If, for example, $P$ has strictly positive density, the induced probabilistic independence model is always a graphoid; see e.g.\ Proposition 3.1 in \cite{lau96}. See also \cite{pet15} for a necessary and sufficient condition for $P$ in order for the intersection property to hold. If the distribution $P$ is a regular multivariate Gaussian distribution, $\mathcal{J}(P)$ is a compositional graphoid; e.g.\ see \cite{stu05}.
Probabilistic independence models with positive densities are not in general compositional; this only holds for special types of multivariate distributions such as, for example,  Gaussian distributions and the symmetric binary distributions used in \cite{wer09}.

Another important property that is not necessarily satisfied by probabilistic independence models is \emph{singleton-transitivity} (also called \emph{weak transitivity} in \cite{pea88}, where it is shown that for Gaussian and binary distributions $P$, $\mathcal{J}(P)$ always satisfies it). For $u$, $v$, and $w$, single elements in $V$,
 \begin{itemize}
	\item[7.] if $u\ci v\cd C$ and $u\ci v\cd C\cup\{w\}$ then $u\ci w\cd C$ or $v\ci w\cd C$ (\emph{singleton-transitivity}).
\end{itemize}
In addition, we have the two following properties:
 \begin{itemize}
\item[8.] if $u\ci v\cd C$ then $u\ci v\cd C\cup\{w\}$ for every $w\in V\setminus \{u,v\}$ (\emph{upward-stability});
\item[9.] if $u\ci v\cd C$ then $u\ci v\cd C\setminus\{w\}$ for every $w\in V\setminus \{u,v\}$ (\emph{downward-stability}).
\end{itemize}
Henceforth, instead of saying that ``$\mathcal{J}(P)$ satisfies these properties'', we simply say that ``$P$ satisfies these properties''. First we provide the following well-known result \cite{sad17}:
\begin{lemma}\label{lem:1}
For a probability distribution $P$ the following holds:
\begin{enumerate}
  \item If $P$ satisfies upward-stability then $P$ satisfies composition.
  \item If $P$ satisfies downward-stability then $P$ satisfies intersection.
\end{enumerate}
\end{lemma}
\subsection{Exchangeability for random vectors and networks}
A probability distribution $P$ over  a finite vector $(X_1, X_2, X_3, \dots, X_n)$ of random variables with the
same shared sample space is \emph{(finitely) exchangeable} if for any permutation $\pi\in S(n)$ of the indices $1, 2, 3, \dots, n$, the probability distribution of the permuted vector $(X_{\pi(1)},X_{\pi(2)},X_\pi{(3)},\dots,X_{\pi(n)})$ is the same as $P$; see \cite{ald85}. We shall for brevity say that the sequence $X$ is exchangeable in the meaning that its distribution is.

We are also concerned with probability distributions on networks that
are finitely exchangeable. A distribution $P$ of a random  matrix $X=(X_{ij})_{i,j\in \nodeset}$  over a finite node set $\nodeset$ with the
same shared sample space is said to be \emph{ (finitely) weakly exchangeable} \citep{sil76,eag78} if  for all permutations $\pi\in S(\nodeset)$ we have that
\begin{equation}
    \mathbb{P}\{(X_{ij}=x_{ij})_{i,j\in \nodeset}\}=\mathbb{P}\{(X_{ij}=x_{\pi(i)\pi(j)})_{i,j\in \nodeset}\}.
\end{equation}
If the matrix $X$ is symmetric --- i.e.\ $X_{ij}=X_{ji}$, we say it is \emph{symmetric weakly exchangeable}. Again, we shall for brevity say that $X$ is weakly or symmetric weakly exchangeable in the meaning that its distribution is.

A symmetric binary array with zero diagonal can be interpreted as a matrix of ties (the \emph{adjacency matrix}) of a random network and, thus, the above concepts can be translated into networks. A random network is \emph{exchangeable} if its adjacency matrix is symmetric weakly exchangeable.  Then it is easy to observe that
 a random network is exchangeable if and only if its distribution is invariant
 under relabeling of the nodes of the network.
\subsection{Undirected and bidirected graphical models}\label{sec:23}

Graphical models \citep[see, e.g.][]{lau96} are statistical models expressing
conditional independence statements among a collection of random variables $X_V = (X_v, v
\in V)$ indexed by a finite set $V$. A graphical model is determined by a graph
$G=(V,E)$ over the
indexing set $V$, and the edge set $E$ (which may include edges of undirected, directed
or bidirected type) encodes conditional independence
relations among the variables, or {\it Markov properties}. 

We say that $C$ \emph{separates} $A$ and $B$ in an undirected graph $G$, denoted by $A\perp_u B\cd C$, if every path between $A$ and $B$ has a vertex in $C$, that is there is no path between $A$ and $B$ outside $C$.  For a bidirected graph $G$, we say that $C$ \emph{separates} $A$ and $B$, denoted by $A\perp_b B\cd C$, if every path between $A$ and $B$ has a vertex outside $C\cup A\cup B$, that is there is no path between $A$ and $B$ within $A\cup B\cup C$.  Note the obvious duality between this and separation for undirected graphs. We might skip the subscripts $u$ and $b$ in $\perp_u$ and $\perp_b$ when it is apparent from the context with which separation we are dealing.

A joint probability
distribution $P$ for $X_V$ is \emph{Markovian} with respect to an undirected graph \citep{dar80} with the vertex set $V$ if $A\perp_u B\cd C$  implies $A\ci B\cd C$. $P$ is \emph{Markovian} with respect to a bidirected graph \citep{cox93,kau96} if $A\perp_b B\cd C$  implies $A\ci B\cd C$.


%



For example, in the undirected graph of Figure \ref{fig:001}(a), 
the global Markov property implies that $\{u,x\}\ci w \cd v$, whereas in the bidirected graph of Figure \ref{fig:001}(b), 
the global Markov property implies that $\{u,x\}\ci w$.

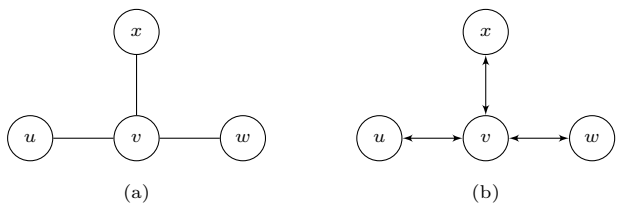
\begin{figure}[htb]
\centering
\begin{tikzpicture}[node distance = 8mm and 8mm, minimum width = 6mm]
    \begin{scope}
      \tikzstyle{every node} = [shape = circle,
      font = \scriptsize,
      minimum height = 6mm,
      inner sep = 2pt,
      draw = black,
      fill = white,
      anchor = center],
      text centered]
      \node(u) at (0,0) {$u$};
      \node(v) [right = of u] {$v$};
      \node(x) [above = of v] {$x$};
			\node(w) [right = of v] {$w$};
      \node(u1) [right = 12mm of w] {$u$};
      \node(v1) [right = of u1] {$v$};
      \node(x1) [above = of v1] {$x$};
			\node(w1) [right = of v1] {$w$};
    \end{scope}
		
    \begin{scope}
    \tikzstyle{every node} = [node distance = 6mm and 6mm, minimum width = 6mm,
    font= \scriptsize,
      anchor = center,
      text centered]
\node(a) [below = 2mm  of v]{(a)};
\node(b) [below = 2mm  of v1]{(b)};

\end{scope}
    \begin{scope}
      \draw (u) -- (v);
      \draw (v) -- (x);
			\draw (v) -- (w);
    \end{scope}
    \begin{scope}[<->, > = latex']
    \draw (u1) -- (v1);
      \draw (v1) -- (x1);
			\draw (v1) -- (w1);
    \end{scope}

    \end{tikzpicture}
		\caption{{\footnotesize (a) An undirected dependence graph. (b) a bidirected dependence graph.}}\label{fig:001}
		\end{figure}



If, for $P$ and undirected $G$, $A\perp_u B\cd C\iff A\ci B\cd C$ then we say that $P$ and $G$ are \emph{faithful}; Similarly, if, for $P$ and bidirected $G$, $A\perp_b B\cd C\iff A\ci B\cd C$ then $P$ and $G$ are \emph{faithful}. Hence, faithfulness implies being Markovian, but not the other way around.

For a given probability distribution $P$, we define the \emph{skeleton} of
$P$, denoted by $\sk(P)$, to be the undirected graph with the vertex set $V$ such
that vertices $u$ and $v$ are not adjacent if and only if there is \emph{some} subset $C$
of $V$ so that $u\cip v\cd C$.  Thus, if $P$ is Markovian with respect to an
undirected graph $G$ then $\sk(P)$ would be a subgraph of $G$ (since for every missing edge $ij$ in $G$, $i\ci j\cd V\setminus\{i,j\}$); and if $P$ is
Markovian with respect to a bidirected graph $G$ then $\sk(P)$ is a subgraph of $\sk(G)$ (since for every missing edge $ij$ in $G$, $i\ci j$).


In general, a graph $G(P)$ is induced by $P$ with skeleton $\sk(P)$. For undirected graphs, let $G_u(P)=\sk(P)$, whereas for bidirected graphs, let $G_b(P)$  be $\sk(P)$ with all edges being bidirected. We shall need the following results from \cite{sad17} (where the first part was first shown in \cite{pea85}):

\begin{prop}\label{prop:und}
Let $P$ be a probability distribution defined over $\{X_\alpha\}_{\alpha\in V}$. It then holds that
\begin{enumerate}
  \item $P$ and $G_u(P)$ are faithful if and only if $P$ satisfies intersection, singleton-transitivity, and upward-stability.
  \item $P$ and $G_b(P)$ are faithful if and only if $P$ satisfies composition, singleton-transitivity, and downward-stability.
\end{enumerate}
\end{prop}
\subsection{Duality in independence models and graphs}\label{sec:dual}
The \emph{dual} of an independence model $\mathcal{J}$ (defined in \cite{mat92} under the name of \emph{dual relation}) is the independence model defined by $\mathcal{J}^d=\{\langle A,B\cd V\setminus (A\cup B\cup C)\rangle:\nn\langle A,B\cd C\rangle\in \mathcal{J}\}$; see also \cite{drt10}. We will need the following lemma regarding the duality of independence models:
\begin{lemma}\label{lem:dual}
For an independence model  $\mathcal{J}$ and its dual $\mathcal{J}^d$,
\begin{enumerate}
  \item $\mathcal{J}$ is semi-graphoid if and only if $\mathcal{J}^d$ is semi-graphoid;
  \item $\mathcal{J}$ satisfies intersection  if and only if $\mathcal{J}^d$ satisfies composition; and vice versa;
  \item $\mathcal{J}$ satisfies singleton-transitivity if and only if $\mathcal{J}^d$ satisfies singleton-transitivity;
  \item  $\mathcal{J}$ satisfies upward-stability  if and only if $\mathcal{J}^d$ satisfies downward-stability; and vice versa.
\end{enumerate}
\end{lemma}
\begin{proof}
1., 2., and 3.\ are proven in \cite{lne07}, which showed that if  $\mathcal{J}$ is singleton-transitive compositional graphoid, so is the dual couple of $\mathcal{J}$ (although this is proven based on the so-called \emph{Gaussoid} formulation). An alternative proof for these,  with the same formulation as in this paper, is found in \cite{mal13}. (In fact the statement in the mentioned article was proven for a generalization of singleton-transitivity, called \emph{dual decomposable transitivity}).

In order to prove 4., suppose that $\mathcal{J}$ satisfies upward-stability, and assume that $\langle i,j\cd C\rangle\in \mathcal{J}^d$ and $k\in C$. Therefore, $\langle i,j\cd V\setminus (\{i,j\}\cup C)\rangle\in \mathcal{J}$. Hence, $\langle i,j\cd V\setminus (\{i,j\}\cup C)\cup\{k\}\rangle\in \mathcal{J}$ because of upward-stability. Therefore  $\langle i,j\cd C\setminus\{k\}\rangle\in \mathcal{J}^d$, which implies downward-stability of $\mathcal{J}^d$. The other direction is similar.
\end{proof}
As an additional statement to the lemma, it can also be shown that $\mathcal{J}$ is closed under \emph{marginalization} if and only if $\mathcal{J}^d$ is closed under \emph{conditioning}; and vice versa; but this result is not needed in this paper.

In addition, for separation in graphs, we will use the following lemma related to duality:
\begin{lemma}\label{lem:5}
Let $G_u$ and $G_b$ be an undirected and a bidirected graph such that $\sk(G_u)=\sk(G_b)$. Then, the independence model $\mathcal{J}(G_b)$ induced by $G_b$ is $\mathcal{J}^d(G_u)$ and vice versa.
\end{lemma}
\begin{proof}
Since the separation satisfies the composition property, it is sufficient to prove the statement for singletons. Suppose that there is a connecting path $\omega$ between $i$ and $j$ given $C$ in $G_u$. This means that no inner vertex of $\omega$ is in $C$; thus they are all in $V\setminus (\{i,j\}\cup C)$. Therefore, in $G_b$, $i$ and $j$ are connecting given $V\setminus (\{i,j\}\cup C)$. The other direction (where $i\notdse j\cd C$ in $G_b\Rightarrow i\notdse j\cd V\setminus (A\cup B\cup C)$ in $G_u$) is proven in a similar way.
\end{proof}

In fact, the second part of Proposition \ref{prop:und}, could be implied by the first part, and vice versa, using Lemmas \ref{lem:dual} and \ref{lem:5}; and the second part of Lemma \ref{lem:1}, could be implied by the first part, and vice versa, using Lemma \ref{lem:dual}.
\section{Results for vector exchangeability}\label{sec:3}
We shall study the relationship between vector exchangeability and conditional independence by using the definitions and results in the previous section. In the entire section, we assume that $P$ is a probability distribution defined over the vector $(X_v)_{v\in V}$. First, notice that exchangeability is closed under marginalization and conditioning:
\begin{prop}\label{prop:3n}
If $X_V$ is an exchangeable random vector then so are the marginal vectors $X_A$, for $A\subseteq V$, and the conditional vectors $X_A\cd X_C=x^*_C$, for disjoint $A,C$ where $V=A\cup C$, if they exist.
\end{prop}
\begin{proof}
First we prove closedness under marginalization: For a permutation matrix $\pi$ of $A$, let the permutation matrix over $V$ be $\pi^*(a)=\pi(a)$, for $a\in A$, and $\pi^*(b)=b$ for $b\in V\setminus A$. The proof then follows from exchangeability of $X_V$.

Now we prove closedness under conditioning: By the definition of conditioning, we need to prove that $P(x_A,x^*_C)=P(x_{\pi(A)},x^*_C)$ for every $\pi$. This is again true by considering $\pi^*(a)=\pi(a)$, for $a\in A$, and $\pi^*(c)=c$ for $c\in C$.
\end{proof}
Notice that the above result implies that the marginal/conditional  $X_A\cd X_C$, (i.e., when $A\cup C\subset V$) is also exchangeable. We now have the following results:
\begin{prop}\label{prop:3}
If $P$ satisfies vector exchangeability then the following holds:
\begin{enumerate}
  \item $P$ satisfies upward-stability if and only if it satisfies composition.
  \item $P$ satisfies downward-stability if and only if it satisfies intersection.
\end{enumerate}
\end{prop}
\begin{proof}
1. ($\Rightarrow$) follows from Lemma \ref{lem:1}. To prove ($\Leftarrow$), let $i\ci j\cd C$. By exchangeability, for an arbitrary $k\notin C\cup\{i,j\}$, we have $i\ci k\cd C$. Composition implies $i\ci j\cup k\cd C$. Weak union implies $i\ci j\cd C\cup\{k\}$.

2. follows from 1.\ and the consequence of duality provided in Lemma \ref{lem:dual}.
\end{proof}

\begin{prop}\label{prop:4}
If $P$ is exchangeable then it satisfies singleton-transitivity.
\end{prop}
\begin{proof}
If $i\ci j\cd C$ and $k\notin C\cup\{i,j\}$ then clearly $i\ci k\cd C$ and $k\ci j\cd C$.
\end{proof}
We see that the skeleton of an exchangeable distribution can only take a very specific form:
\begin{prop}\label{prop:5}
If $P$ is exchangeable then $\sk(P)$ is either an empty or a complete graph.
\end{prop}
\begin{proof}
If there is any independence statement of form $i\ci j\cd C$ then by permutation for all variables, we obtain $k\ci l\cd C'$  for all $k$ and $l$. Therefore, $\sk(P)$ is empty. If there is no independence statement of this form then $\sk(P)$ is complete.
\end{proof}

Notice that for faithfulness to empty or complete graphs, it is immaterial
whether one considers undirected or bidirected interpretation of graphs.
\begin{coro}
Let $P$ satisfy vector exchangeability. If $P$ is faithful to a graph (both under undirected interpretation and under bidirected
interpretation) then the graph is empty or complete.
\end{coro}
Hence, there are two regimes available: if there is no independence statement implied by $P$ then we are in the complete graph regime; and if there is at least one conditional independence statement implied by $P$ then we are in the empty graph regime. The following also provides conditions for the opposite direction of the above result:
\begin{theorem}
If $P$ is exchangeable then $P$ is faithful to a graph (both under undirected interpretation and under bidirected
interpretation) if and only if $P$ satisfies the intersection and composition properties. The graph must then be either empty or complete.
\end{theorem}
\begin{proof}
The first result follows from Propositions \ref{prop:und}, \ref{prop:3}, and \ref{prop:4}. The second follows from Proposition \ref{prop:5}.
\end{proof}
Indeed other exchangeable distributions may exist but they are not faithful to a graph. We then have the following corollaries:
\begin{coro}
 Let $P$ be exchangeable and there exists an independence statement induced by $P$. It then holds that  all variables $X_v$ are completely independent of each other if and only if $P$ satisfies intersection and composition.
\end{coro}
\begin{coro}
Let $P$ be a regular exchangeable Gaussian distribution. If there is a zero element in its covariance matrix then all $X_v$ are completely independent; and otherwise they are completely dependent.
\end{coro}
\begin{proof}
The proof follows from the fact that a regular Gaussian distribution satisfies the intersection and composition properties.
\end{proof}
Notice that the above statement could be shown otherwise since a zero off-diagonal entry of the covariance matrix can be permuted by exchangeability to every other off-diagonal entry, making the covariance matrix diagonal.
\begin{prop}
If $P$ is exchangeable, satisfies intersection (this holds when $P$ has a positive density), and if there exists one independence statement induced by $P$ then all variables $X_v$ are marginally independent of each other.
\end{prop}
\begin{proof}
By Proposition \ref{prop:3}, $P$ satisfies downward-stability. An independence statement $A\ci B\cd C$, by the use of decomposition implies $i\ci j\cd C$ for an arbitrary $i\in A$ and $j\in B$.  Downward-stability implies $i\ci j$. Exchangeability implies all pairwise marginal independences.
\end{proof}
\begin{example}
Consider an exchangeable distribution $P$ over four variables $(i,j,k,l)$ with $i\ci j\cd k$. By exchangeability, all independences of form $\pi(i)\ci \pi(j)\cd \pi(k)$ hold for any permutation $\pi$ on  $(i,j,k,l)$. It is easy to see that none of the semi-graphoid axioms can generate new independence statements from these. If intersection holds then, for example, from $i\ci j\cd k$ and $i\ci k\cd j$, we obtain $i\ci \{j,k\}$, which by decomposition implies $i\ci j$, and hence all marginal independences between singletons.  If composition holds then, for example, from $i\ci j\cd k$ and $i\ci l\cd k$, we obtain $i\ci \{j,l\}\cd k$, which by weak union implies $i\ci j\cd \{k,l\}$, and hence all conditional independences between singletons given the remaining variables.
\end{example}

\section{Results for exchangeability for random networks}\label{sec:4}
Henceforth, in the context of random networks, we consider vectors whose components are indexed by dyads (i.e.\ two-element subsets of the node set $\mathcal{N}$). Thus,
for a conditional independence statement of the form $A\ci B \cd C$, $A,B,C \subset \mathcal{D}(\mathcal{N})$ are
pairwise disjoint subsets of dyads. (Later on, we particularly write the conditioning sets as simply $C$, which should be considered a subset of dyads.) Notice that we simply use the notation $ij$ for a dyad whose endpoints are nodes $i$ and $j$. This is different from the notation $\{i,j\}$, which indicates the set of two nodes $i$ and $j$ as used in the previous section.
%
\subsection{Marginalization and conditioning for exchangeable random networks}
Here we focus on marginalization over and conditioning on arbitrary sets of dyads. Notice that by \emph{marginalizing over} a set $M$, we mean we marginalize the set $M$ out, which results in a distribution over $\mathcal{D}(\nodeset)\setminus M$.
However, exchangeable networks are not always closed under marginalization over or conditioning on an arbitrary set of dyads:

For a marginal network $X_A$, where $A$ is a subset of dyads, exchangeability and summing up all probabilities over values of the dyads that are marginalized over imply that $P(X_A=x_A)=P((X_{\pi(i)\pi(j)})_{ij\in A}=x_A)$, for any permutation $\pi$. However, this is not necessarily equal to $P(X_A=(x_{\pi(i)\pi(j)})_{ij\in A})$, which is what we need for exchangeability of the marginal to hold. For conditioning, in fact, $X_A\cd X_C$ is not exchangeable if a node appears in a dyad in $A$ and a dyad in $C$ (e.g.\ $i$ appearing in $ij\in A$ and $ik\in C$). This is because, for a permutation $\pi$ that maps $i$ to a node other than $i$, exchangeability of $X_A\cd X_C$ is equivalent to $P(x_A\cd x_C)=P((x_{\pi(i)\pi(j)})_{ij\in A}\cd x_C)$, which itself is equivalent to $P(x_A,x_C)=P((x_{\pi(i)\pi(j)})_{ij\in A},x_C)$. But, this does not necessarily hold as $i$ is mapped to another node in $A$ but not in $C$.

However, we have the following:
\begin{prop}
Let $A$ and $C$  be disjoint subsets of dyads of an exchangeable random network $X$ such that $A$ and $C$ do not share any nodes, i.e.\ if $ij\in A$ then there is no dyad $ik$ or $jk$ in $C$ for any node $k\in\nodeset$. It then holds that the conditional/marginal random network $X_A\cd X_C$ is exchangeable.
\end{prop}
\begin{proof}
Let $\nodeset(A\cup C)$ be the set of all endpoints of dyads in $A$ and $C$, and define similarly $\nodeset(A)$ and $\nodeset(C)$. Define the permutation $\pi^*\in S(\nodeset(A\cup C))$ such that $\pi^*(i)=\pi(i)$ for $i\in \nodeset(A)$ and $\pi^*(k)=k$ for $k\in \nodeset(C)$. Notice that this is well-defined since $A$ and $C$ do not share any nodes. Using $\pi^*$ and by exchangeability of $X$, we conclude that $P(x_A,x_C)=P((x_{\pi(i)\pi(j)})_{ij\in A},x_C)$, which, as mentioned before, is equivalent to the exchangeability of $X_A\cd X_C$.
\end{proof}

\subsection{Types of graphs faithful to exchangeable distributions}\label{sec:41}
An analogous result to that of vector exchangeability, concerning the skeleton of an exchangeable probability distribution, was proven in \cite{lau17}:
\begin{prop}\label{prop:7}
If a distribution $P$ over a random network $X$ is exchangeable then $\sk(P)$ is one of the following:
\begin{enumerate}
  \item the empty graph;
  \item the incidence graph;
  \item the complement of the incidence graph;
  \item the complete graph.
\end{enumerate}
\end{prop}
Notice that a pairwise independence statement for an exchangeable $P$ over a random network $X$ is of form $ij\ci kl\cd C$, where $i,j,k,l$ are nodes of $X$. Depending on the type of $\sk(P)$, these statements take different forms.
\begin{lemma}\label{lem:2}
Suppose that there exists an independence statement of form $ij\ci kl\cd C$, $i\neq j,k\neq l$, for an exchangeable $P$ over a random network $X$. It is then in one  of the following forms depending on the type of $\sk(P)$:
\begin{enumerate}
  \item empty graph $\Rightarrow$ no constraints on $i,j,k,l$;
  \item incidence graph $\Rightarrow$ $i\neq l,k$ and $j\neq l,k$;
  \item complement of the incidence graph $\Rightarrow$ $i=k$ or $i=l$ or $j=k$ or $j=l$;
\end{enumerate}
if $\sk(P)$ is the complete graph then it is not possible to have such a conditional independence statement.
\end{lemma}
\begin{proof}
The proof follows from the fact if there is an edge between $ij$ and $kl$ in $G(P)$ then there is no statement of form $ij\ci kl\cd C$.
\end{proof}
It is clear that if $\sk(P)$ is the complete graph then every dyad is completely dependent on every other dyad. Thus we consider the cases of the incidence graph skeleton and the complement of the incidence graph skeleton separately. However, before this, we show that among the known graphical models, only undirected and bidirected graphs can be faithful to an exchangeable distribution.

In order to do so, we can start off by considering any class of \emph{mixed graphs}, i.e.\ graphs with simultaneous undirected, directed, or bidirected edges that use the (unifying) separation criterion introduced in \cite{sadl16}. To the best of our knowledge, the largest class of such graphs is the class of \emph{chain mixed graphs} \cite{sadl16}, which includes the classes of \emph{ancestral graphs} \cite{ric02}, \emph{LWF} \cite{lau89} and \emph{regression chain graphs} \cite{wer94}, and several others; see \cite{sadl16}. We require some definitions and results, including the formulation of the separation criterion, which we only need for the results in this subsection. We provide these together with the proof of the main result (Theorem \ref{thm:2}) in Appendix \ref{app}. If the reader is only interested in the statement of the theorem and not the proof, we suggest that they skip the material in the appendix.

Two graphs are called \emph{Markov equivalent} if they induce the same independence model.
\begin{theorem}\label{thm:2}
If a distribution $P$  over an exchangeable random network with $n$ nodes is faithful to a chain mixed graph $G$ then $G$ is Markov equivalent to one of the following graphs:
\begin{enumerate}
  \item the empty graph;
  \item  the undirected incidence graph, $L_{-}(n)$;
  \item  the bidirected incidence graph, $L_{\leftrightarrow}(n)$;
  \item   the undirected complement of the incidence graph, $L^c_{-}(n)$;
  \item  the bidirected complement of the incidence graph, $L^c_{\leftrightarrow}(n)$;
  \item the complete graph.
\end{enumerate}
\end{theorem}
Hence, for exchangeable random networks, there are six regimes available, where these are three pairs that are complement of each other. Within the two non-trivial pairs, the two undirected and bidirected cases act as dual of each other as described in Section \ref{sec:dual}. We will make use of this duality to simplify the results and proofs.

Although the following method is not unique, here we provide a simple test to decide in which regime a given exchangeable distribution lies:
\begin{alg}\label{alg:1}
For arbitrary fixed nodes $i,j,k,l,m$ of a given exchangeable network, test the following:
\begin{itemize}
 \item $ij\ci kl\cd C$, for some $C$, and $ij\ci ik\cd C'$, for some $C'\Rightarrow$ Empty graph;
 \item $ij\ci kl\cd C$, for some $C$, and $ij\notci ik\cd C'$, for any $C'\Rightarrow L(n)$:
 \begin{itemize}
 \item $ik\in C\Rightarrow L_{-}(n)$;
 \item $ik\notin C\Rightarrow L_{\leftrightarrow}(n)$;
\end{itemize}
 \item $ij\notci kl\cd C$, for any $C$, and $ij\ci ik\cd C'$, for some $C'\Rightarrow L^c(n)$:
 \begin{itemize}
 \item $lm\in C'\Rightarrow L^c_{-}(n)$;
 \item $lm\notin C'\Rightarrow L^c_{\leftrightarrow}(n)$;
\end{itemize}
 \item $ij\notci kl\cd C$, for any $C$, and $ij\notci ik\cd C'$, for any $C'\Rightarrow$ Complete graph.
\end{itemize}
\end{alg}
\begin{prop}
If a distribution $P$  over an exchangeable random network is faithful to a graph $G$ then Algorithm \ref{alg:1} determines the Markov equivalence class of $G$.
\end{prop}
\begin{proof}
First, we show that this test covers all the possible cases: The first level tests ($ij\ci kl\cd C$ and $ij\ci ik\cd C'$) clearly cover all the cases concerning the skeleton of the graph. But, the second level test ($ik\in C$ or $lm\in C'$), which only concerns the non-trivial skeletons, might not be consistent: For example, first consider the $L(n)$ case, and assume that $ij\ci kl\cd C_1$ and $ij\ci kl\cd C_2$, for some $C_1,C_2$, but $ik\in C_1$ and $ik\notin C_2$. However, in such cases, it is easy to show that 
$P$ cannot be faithful to either of the two undirected or bidirected graphs. The case of $L^c(n)$ is similar.

The algorithm also outputs the correct regimes: The tests of $ij\notci kl\cd C$ and $ij\ci ik\cd C'$ clearly determine the skeleton $\sk(P)$. The test $ik\in C$ then determines the type of edges since, in $\langle ij,ik,kl\rangle$, if $ik\notin C$ then $ij$ and $kl$ cannot be separated in the undirected graph; and if $ik\in C$ then $ij$ and $kl$ cannot be separated in the bidirected graph. The test $lm\in C'$ can be proven similarly.
\end{proof}
Unlike the vector exchangeability case, the intersection and composition properties are not in general sufficient for faithfulness of exchangeable network distributions to the graphs provided in Theorem \ref{thm:2}. However, in special cases this holds. We detail this below for each regime mentioned above.
\subsection{The incidence graph case}
In this section, we assume that $\sk(P)=L_{-}(n)$. The following example shows that, in principle, intersection and composition are not sufficient for faithfulness of $P$ and $L_{-}(n)$ (and $P$ and the bidirected incidence graph $L_{\leftrightarrow}(n)$).
\begin{example}\label{ex:2}
Suppose that there is an exchangeable $P$ that induces $ij\ci kl\cd C_{ij,kl}$, where  $C_{ij,kl}=\{ik,il,jk,jl\}$. Exchangeability implies that $ij\ci kl\cd C_{ij,kl}$ for all $i,j,k,l$. Suppose, in addition, that $P$ satisfies upward-stability. Notice that here we do not show that such a probability distribution necessarily exists -- one can treat this defined independence model as an example of ``a network-exchangeable semi-graphoid".

It is easy to see that $\sk(P)= L_{-}(n)$. Moreover, by Lemma \ref{lem:1}, $P$ satisfies composition. 
 In addition, $P$ satisfies intersection: Notice that none of the semi-graphoid axioms plus upward-stability and composition imply an independence statement of form $A\ci B\cd C$ where $A,B$ both contain a node $i$. In addition, it can be seen that these axioms imply that if there is $A\ci B\cd C$ then for every $ij\in A$ and $kl\in B$, it holds that $C_{ij,kl}\subseteq C$. Let $C_{A,B}=\bigcup_{ij\in A,kl\in B}C_{ij,kl}$. By these two observations, we conclude that if $A\ci B\cd C\cup D$ and $A\ci D\cd C\cup B$ then $C_{A,D}\subseteq C$. By upward-stability all statements of form  $ij\ci kl\cd C$ where $C_{ij,kl}\subseteq C$ hold. Hence, by composition, we have that  $A\ci B\cd C$. Hence, by contraction,  $A\ci B\cup D\cd C$.

However, $P$ does not satisfy singleton-transitivity: Consider $ij\ci kl\cd C_{ij,kl}$ and $ij\ci kl\cd C_{ij,kl}\cup\{km\}$. If, for contradiction, singleton-transitivity holds then, because $\sk(P)=L_{-}(n)$, we have that $ij\ci km\cd C_{ij,kl}$. But, it can be seen that this statement is not in the independence model since no compositional graphoid axioms can generate this statement from $\{ij\ci kl\cd C_{ij,kl}\}$.

By Proposition \ref{prop:und} (1.), it is implied that, although intersection and composition are satisfied, $P$ is not faithful to $L_{-}(n)$.

If we now suppose that $P$ induces $ij\ci kl\cd C^d_{ij,kl}$, where $C^d_{ij,kl}=V\setminus (C_{ij,kl}\cup\{ij,kl\})$, $P$ is exchangeable, and it satisfies downward-stability then, by using the duality (Lemmas \ref{lem:dual} and \ref{lem:5}), we conclude that although intersection and composition are satisfied, $P$ is not faithful to $L_{\leftrightarrow}(n)$.
\end{example}
However, under certain assumptions, intersection and composition are sufficient for faithfulness. We will consider the two dual regimes within the incidence graph case related to the cases of Theorem \ref{thm:2}: the undirected incidence graph case and the bidirected incidence graph case. We make use of the duality between these to immediately extend the results in the undirected case to the bidirected case.

As in Example \ref{ex:2}, let $C_{ij,kl}=\{ik,il,jk,jl\}$. For the faithfulness results to the undirected case, one assumption that is used is that for some (and because of exchangeability for all) $i,j,k,l$, $ij\ci kl\cd C$ implies that $C_{ij,kl}\subseteq C$. For the faithfulness results to the bidirected case, one assumption is that for some (and because of exchangeability for all) $i,j,k,l$, $ij\ci kl\cd C$ implies that $C_{ij,kl}\cap C=\varnothing$.
%


For a separation statement $A\dse B\cd C$, we define $C$ to be a \emph{minimal} separator in the case that if we remove any vertex from $C$, the separation does not hold; we define a \emph{maximal} separator similarly.  Let also $C_{ij}=\{ir,jr:\forall r\neq i,j\}$ and  $C^d_{ij}=V\setminus (C_{ij}\cup \{ij\})=\{lm:\forall l,m\notin\{i,j\}\}$. It holds that $C_{ij}$ is a minimal separator of $ij,kl$ in $L_{-}(n)$ and $C^d_{ij}\setminus\{kl\}$ is a maximal separator  in $L_{\leftrightarrow}(n)$:
\begin{prop}\label{prop:12}

\begin{enumerate}
  \item In $L_{-}(n)$, it holds that $ij\dse kl\cd C_{ij}$, and if $ij\dse kl\cd C$  then $|C_{ij}|\leq |C|$. In fact, if $C\neq C_{ij}$ and $C\neq C_{kl}$ then $|C_{ij}|< |C|$.
  \item In $L_{\leftrightarrow}(n)$, it holds that $ij\dse kl\cd C^d_{ij}\setminus\{kl\}$, and if $ij\dse kl\cd C$  then $|C^d_{ij}|\geq |C|+1$. In fact, if $C\neq C^d_{ij}\setminus\{kl\}$ and $C\neq C^d_{kl}\setminus\{ij\}$ then $|C^d_{ij}|> |C|+1$.
\end{enumerate}
\end{prop}
\begin{proof}
1.\ The first claim is straightforward to prove since every path from $kl$ to $ij$ must pass through an adjacent vertex of $ij$, which contains $i$ or $j$. To prove the second statement, notice that if $im\notin C$ then at least $km$ and $lm$ must be in $C$. The same vertices $km$ and $lm$ appear if $jm\notin C$ too, but for no other vertices of $C_{ij}$ missing in $C$. Thus, for every missing member of $C_{ij}$ in $C$, there is at least a member of $C_{kl}$ that should be in $C$. Hence, $|C_{ij}|\leq |C|$.

To prove the third statement, suppose, for contradiction, that $C\neq C_{ij}$ and $C\neq C_{kl}$ and $|C|=|C_{ij}|$. Using the fact that, for every missing member of $C_{ij}$ in $C$, there is at least a member of $C_{kl}$ that should be in $C$, there cannot be any vertex outside $C_{ij}\cup C_{kl}$ in a $C$. Hence, $C\subset C_{ij}\cup C_{kl}$. If $n>5$ and, say, $im,jm\notin C$ but $km,lm\in C$ then consider the vertices $ih,jh,kh,lh$. Without loss of generality, assume that $ih,jh\in C$ but $kh,lh\notin C$. Then the path $\langle kl,kh,mh,jm,ij\rangle$ connects $kl$ and $ij$, a contradiction. The cases where $n=3,4,5$ are easy to check.

2.\ The proof follows from 1.\ by using the duality (Lemma \ref{lem:5}), and observing that $|C^d_{ij}\setminus\{kl\}|=|C^d_{ij}|-1$.
\end{proof}
However, not all minimal separators of $ij,kl$ in $L_{-}(n)$ are of the form above. For example, in $L_{-}(6)$, consider the set $C=\{ik,il,im,jk,jl,jm,hk,hl,mh\}$. It holds that $ij\dse kl\cd C$. However, $C$ is not of the form $C_{pq}$ for any pair $p,q\in\{i,j,k,l,m,h\}$. The same can be said about $L_{-}(n)$: Consider $C^d=V\setminus (C\cup \{ij,kl\})$. By Lemma \ref{lem:5}, we have that $ij\dse kl\cd C^d$ in $L_{\leftrightarrow}(6)$, and, in addition, $C^d$ is maximal. 
\begin{prop}\label{prop:3n0}
Let a distribution $P$ be defined over an exchangeable random network and $\sk(P)=L_{-}(n)$, and consider some nodes $i,j,k,l$.
\begin{enumerate}
  \item Suppose that for every minimal $C$ such that $ij\ci kl\cd C$, it holds that $C_{ij,kl}\subseteq C$ and $C$ is invariant under swapping $k$ and $m$, and $l$ and $h$, for every $m,h$, $m\neq h$, $mh\notin C\cup\{ij,kl\}$. It then holds that if $P$ satisfies composition then it satisfies upward-stability and singleton-transitivity.
  \item Suppose that for every maximal $C$ such that $ij\ci kl\cd C$, it holds that $C_{ij,kl}\cap C=\varnothing$ and $C$ is invariant under swapping $k$ and $m$, and $l$ and $h$, for every $m,h$, $m\neq h$, $mh\in C$. It then holds that if $P$ satisfies intersection then it satisfies downward-stability and singleton-transitivity.
\end{enumerate}
\end{prop}
\begin{proof}
By Lemma \ref{lem:2}, we know that $i,j,k,l$ are all different.

1.\ First, we prove upward-stability. Suppose that $ij\ci kl\cd C$. Notice that, because of exchangeability, the assumptions of the statement hold for every $i,j,k,l,m,h$. For a variable $mh\notin C\cup\{ij,kl\}$, we prove that $ij\ci kl\cd C\cup\{mh\}$ by induction on $|C|$. Notice that, by assumption, $mh\notin C_{ij,kl}$. In addition, without loss of generality, we can assume that $m,h\neq i,j$ since otherwise we can swap $i$ and $k$ and $j$ and $l$ and proceed as follows, and finally swap them back.

%

The base case is when $C$ is minimal. Because of exchangeability, by the invariance of $C$ under the aforementioned swaps, and since $mh\notin C_{ij,kl}$, we have that $ij\ci mh\cd C$. By composition $ij\ci \{kl,mh\}\cd C$, which, by weak union, implies $ij\ci kl\cd C\cup\{mh\}$.

Now suppose that $ij\ci kl\cd C$, $C$ is not minimal, and, for every  $C'$ such that $|C'|<|C|$,  $[ij\ci kl\cd C'\Rightarrow ij\ci kl\cd C'\cup\{op\}]$, for every $op$. Consider the independence statement $ij\ci kl\cd C_0$ such that $C_0\subset C$ and $C_0$ is minimal. We have that $ij\ci mh\cd C_0$. By induction hypothesis, we can add vertices to the conditioning set in order to obtain $ij\ci mh\cd C$.
Now, again composition and weak union imply the result.

Singleton-transitivity also follows from the above argument. We need to show that $ij\ci kl\cd C$ and $ij\ci kl\cd C\cup\{mh\}$ imply $ij\ci mh\cd C$ or $mh\ci kl\cd C$. (Notice that because of upward-stability $ij\ci kl\cd C\cup\{mh\}$ is immaterial.) Again without loss of generality, we can assume that $m,h\neq i,j$, and the above argument showed that $ij\ci mh\cd C$.

2.\ The proof follows from part 1.\ and the duality (Lemma \ref{lem:dual}).
%
%
%
%
%
\end{proof}
\begin{theorem}
Let a distribution $P$ be defined over an exchangeable random network and $\sk(P)=L_{-}(n)$, and consider some nodes $i,j,k,l$.
\begin{enumerate}
  \item Suppose that for every minimal $C$ such that $ij\ci kl\cd C$, it holds that $C_{ij,kl}\subseteq C$ and $C$ is invariant under swapping $k$ and $m$, and $l$ and $h$, for every $m,h$, $m\neq h$, $mh\notin C\cup\{ij,kl\}$. Then $P$ is faithful to $L_{-}(n)$ if and only if $P$ satisfies the intersection and composition properties.
  \item  Suppose that for  every maximal $C$ such that $ij\ci kl\cd C$, it holds that $C_{ij,kl}\cap C=\varnothing$ and $C$ is invariant under swapping $k$ and $m$, and $l$ and $h$, for every $m,h$, $m\neq h$, $mh\in C$. Then $P$ is faithful to $L_{\leftrightarrow}(n)$ if and only if $P$ satisfies the intersection and composition properties.
\end{enumerate}
\end{theorem}
\begin{proof}
The proof follows from Propositions \ref{prop:3n0} and \ref{prop:und}.
\end{proof}

\subsection{The complement of the incidence graph case}
In this section, we assume that $\sk(P)=L^c_{-}(n)$. We show again that, under certain assumptions, intersection and composition are sufficient for faithfulness. We again utilize the duality of the two additional regimes in the complement of the incidence graph case related to the cases of Theorem \ref{thm:2}: the undirected complement of the incidence graph case and the bidirected complement of incidence graph case. Let $C^d_{ijk}=\{lm:\forall l,m\notin\{i,j,k\}\}$. For the faithfulness results to the, respectively, undirected and bidirected case, an assumption that is used is that for some (and because of exchangeability for all) $i,j,k$, $ij\ci ik\cd C$ implies that $C^d_{ijk}\subseteq C$ for undirected case and  $C^d_{ijk}\cap C=\varnothing$ for the bidirected case. 

Let $C_j=\{jr:\forall r\neq j\}$. Recall also that $C_{ij}=\{ir,jr:\forall r\neq i,j\}$ and $C^d_{ij}=\{lm:\forall l,m\notin\{i,j\}\}$. $C^d_{ij}$ is a minimal separator of $ij,ik$ in $L^c_{-}(n)$, and $C_{ij}\setminus\{ik\}$ is a maximal separator in $L^c_{\leftrightarrow}(n)$:
\begin{prop}\label{prop:12n}
Let $n>4$.
\begin{enumerate}
  \item In $L^c_{-}(n)$, it holds that $ij\dse ik\cd C^d_{ij}$, and if $ij\dse ik\cd C$  then $|C^d_{ij}|\leq |C|$. In fact, if $C\neq C^d_{ij}$ and $C\neq C^d_{ik}$ then $|C^d_{ij}|< |C|$.
  \item In $L^c_{\leftrightarrow}(n)$, $n>4$, it holds that $ij\dse ik\cd C_{ij}\setminus\{ik\}$, and if $ij\dse ik\cd C$  then $|C_{ij}|\geq |C|+1$. In fact, if $C\neq C_{ij}\setminus\{ik\}$ and $C\neq C^d_{ik}\setminus\{ij\}$ then $|C^d_{ij}|> |C|+1$.
\end{enumerate}
\end{prop}
\begin{proof}
1.\ The first claim is straightforward to prove since every path from $ik$ to $ij$ must pass through an adjacent vertex of $ij$, which does not contain $i$ or $j$. This set is $C^d_{ij}$. To prove the second statement, consider the sets $C^d_{ij}\setminus C^d_{ik}=C_k\setminus\{ik,jk\}$ and $C^d_{ik}\setminus C^d_{ij}=C_j\setminus\{ij,jk\}$, i.e., the neighbours of $ij$ and $ik$ with the joint neighbours removed. For every subset $S$ of $C_k\setminus\{ik,jk\}$, clearly there are at least the same number of vertices in $C_j\setminus\{ij,jk\}$ adjacent to members of $S$. Hence, the Hall's marriage theorem \cite{hal35} implies the result.

To prove the third statement, suppose, for contradiction, that $C\neq C^d_{ij}$ and $C\neq C^d_{ik}$ and $|C|=|C^d_{ij}|$. Using the fact that, for every missing member of $C^d_{ij}$ in $C$, there is at least a member of $C^d_{ik}$ that should be in $C$, there cannot be any vertex outside $C^d_{ij}\cup C^d_{ik}$ in a $C$. Hence, $C\subset C^d_{ij}\cup C^d_{ik}$. If $n>4$ and, $km,jo\notin C$, where $o\neq m$,  the path $\langle ik,jo,km,ij\rangle$ connects $ik$ and $ij$, a contradiction. Thus, say, $km,jm\notin C$ (which implies $jl,kl\in C$). Then the path $\langle ik,jm,il,km,ij\rangle$ connects $ik$ and $ij$, a contradiction.

2.\ The proof follows from the previous part and Lemma \ref{lem:5}, and observing that $|C_{ij}\setminus\{ik\}|=|C_{ij}|-1$.
\end{proof}
However, not all minimal separators of $ij,ik$ in $L^c_{-}(n)$ are of the form above. For example, in $L^c_{-}(5)$, consider the set $C=\{kl,jl,il,lm\}$. It holds that $ij\dse ik\cd C$. In addition, $C$ is minimal. However, $C$ is not of the form in the above proposition. In $L^c_{\leftrightarrow}(5)$, for the set $C^d=V\setminus (C\cup\{ij,ik\})$, by using Lemma \ref{lem:5}, we see that $ij\dse ik\cd C^d$, and $C^d$ is maximal.
\begin{prop}\label{prop:3nn0}
Let a distribution $P$ be defined over an exchangeable random network and $\sk(P)=L_{-}^c(n)$, and consider some nodes $i,j,k$.
\begin{enumerate}
  \item Suppose that for every minimal $C$ such that $ij\ci ik\cd C$,  $C^d_{ijk}\subseteq C$ and $C$ is invariant under swapping $k$ and $m$ for every $m$ with an $l$ such that $lm\notin C$. It then holds that if $P$ satisfies composition then it satisfies upward-stability and singleton-transitivity.
  \item Suppose that for  every maximal $C$ such that $ij\ci ik\cd C$, $C^d_{ijk}\cap C=\varnothing$ and $C$ is invariant under swapping $k$ and $m$  for every $m$ with an $l$ such that $lm\in C$. It then holds that if $P$ satisfies intersection then it satisfies downward-stability and singleton-transitivity.
\end{enumerate}
\end{prop}
\begin{proof}
By Lemma \ref{lem:2}, we know that the form of independencies for $\sk(P)=L^c_{-}(n)$ is $ij\ci ik\cd C$, as provided in the statement of the proposition.

1.\ First, we prove upward-stability. Suppose that $ij\ci ik\cd C$. Notice that, because of exchangeability, the assumptions of the statement hold for every $i,j,k,l,m$. For a variable $lm\notin C\cup\{ij,ik\}$, we prove that $ij\ci ik\cd C\cup\{lm\}$ by induction on $|C|$. Notice that, by assumption, $lm\in C_{ijk}$. 
Without loss of generality, we can assume that $l\in\{i,j,k\}$, and further, $l=i$ or $l=k$ since otherwise we can swap $j$ and $k$ and proceed as follows, and finally swap them back.

The base case is when $C$ is minimal. We have two cases: If $l=i$ then because of exchangeability, by the invariance of $C$ under swapping $k$ and $m$, and since $im\notin C^d_{ijk}$, we have that $ij\ci im\cd C$. By composition $ij\ci \{ik,im\}\cd C$, which, by weak union, implies $ij\ci ik\cd C\cup\{im\}$.  If $l=k$ then by swapping $k$ and $i$, we have that $jk\ci ik\cd C$. Now, notice that the statement $ij\ci ik\cd C$ does not change under the $kj$-swap, which implies that $C$ is also invariant under swapping $j$ and $m$. By this swap, we have $km\ci ik\cd C$. By this, $ij\ci ik\cd C$, and the use of composition, we obtain $\{ij,km\}\ci ik\cd C$, which, by weak union, implies $ij\ci ik\cd C\cup\{km\}$.

Now suppose that $ij\ci ik\cd C$, $C$ is not minimal, and, for every  $C'$ such that $|C'|<|C|$,  $[ij\ci ik\cd C'\Rightarrow ij\ci ik\cd C'\cup\{op\}]$, for every $op$. Consider the independence statement $ij\ci ik\cd C_0$ such that $C_0\subset C$ and $C_0$ is minimal. We have that $ij\ci im\cd C_0$ or $km\ci ik\cd C_0$. By induction hypothesis, we can add vertices to the conditioning set in order to obtain $ij\ci im\cd C$ or $km\ci ik\cd C$.
Now, again composition and weak union imply the result.

Singleton-transitivity also follows from the above argument. We need to show that $ij\ci ik\cd C$ and $ij\ci ik\cd C\cup\{lm\}$ imply $ij\ci lm\cd C$ or $lm\ci ik\cd C$. (Notice that because of upward-stability $ij\ci ik\cd C\cup\{lm\}$ is immaterial.) Again without loss of generality,  we can assume that $l=i$ or $l=k$, and the above argument showed that $ij\ci im\cd C$ or $km\ci ik\cd C$.

2.\ The proof follows from part 1.\ and the duality (Lemma \ref{lem:dual}).
\end{proof}
\begin{theorem}
Let a distribution $P$ be defined over an exchangeable random network and $\sk(P)=L^c_{-}(n)$, and consider some nodes $i,j,k$.
\begin{enumerate}
  \item Suppose that for every minimal $C$ such that $ij\ci ik\cd C$,  $C^d_{ijk}\subseteq C$ and $C$ is invariant under swapping $k$ and $m$ for every $m$ with an $l$ such that $lm\notin C$. Then $P$ is faithful to $L^c_{-}(n)$ if and only if $P$ satisfies the intersection and composition properties.
  \item  Suppose that for  every maximal $C$ such that $ij\ci ik\cd C$, $C^d_{ijk}\cap C=\varnothing$ and $C$ is invariant under swapping $k$ and $m$  for every $m$ with an $l$ such that $lm\in C$. Then $P$ is faithful to $L^c_{\leftrightarrow}(n)$ if and only if $P$ satisfies the intersection and composition properties.
\end{enumerate}
\end{theorem}
\begin{proof}
The proof follows from Propositions \ref{prop:3nn0} and \ref{prop:und}.
\end{proof}

\subsection{The empty graph case}
Clearly every minimal separator in the empty graph is the empty set, and the maximal separator is all the remaining vertices. In the case where $\sk(P)$ is the empty graph, we have the following.
\begin{prop}\label{prop:3nnnn0}
Let a distribution $P$ be defined over an exchangeable random network, and $\sk(P)$ be the empty graph, and consider some nodes $i,j,k,l$.
\begin{enumerate}
  \item Suppose that $ij\ci kl$ and $ij\ci ik$ hold. It then holds that if $P$ satisfies composition then it satisfies upward-stability and singleton-transitivity.
  \item Suppose that $ij\ci kl\cd V\setminus\{ij,kl\}$ and $ij\ci ik\cd V\setminus\{ij,ik\}$ hold. It then holds that if $P$ satisfies intersection then it satisfies downward-stability and singleton-transitivity.
\end{enumerate}
\end{prop}
\begin{proof}
By Lemma \ref{lem:2}, we know that $i,j,k,l$ are disjoint.

1.\ First, we prove upward-stability. Suppose that $ij\ci kl\cd C$ or $ij\ci ik\cd C$. Notice that, because of exchangeability, the assumptions of the statement hold for $i,j,k,l$ that are all different. We prove that $ij\ci kl\cd C\cup\{mh\}$ or $ij\ci ik\cd C\cup\{mh\}$, for every $mh\notin C\cup\{ij,kl\}$ or $mh\notin C\cup\{ij,ik\}$, respectively, by induction on $|C|$.

The base case is when $C=\varnothing$. First, consider the case where $ij\ci kl$. If $mh\notin C_{ij,kl}$ then by swapping $k$ and $m$ and $l$ and $h$, we obtain $ij\ci mh$. If $mh\in C_{ij,kl}$ then say $mh=jl$. We use $ij\ci ik$ and first swap $i$ and $j$ to obtain $ij\ci jk$. Now we swap $k$ and $l$ to obtain $ij\ci jl$. (The other three cases of $mh$ are similar.) Now composition and weak-union imply the result.

Now, consider the case where $ij\ci ik$. First suppose that $mh\in C_{ij,kl}$. If $mh=il$ then by swapping $j$ and $l$, we obtain $il\ci ik$. If $mh=jk$ then by swapping $i$ and $k$, we obtain $jk\ci ik$. If $mh=jl$ then by swapping $i$ and $j$ and then $k$ and $l$, we obtain $ij\ci jl$. If $mh\notin C_{ij,kl}$ then use $ij\ci kl$ and swap $k$ and $m$ and $l$ and $h$ to obtain $ij\ci mh$. Now composition and weak-union imply the result.

The inductive step is similar to the inductive step of Proposition \ref{prop:3n0} or Proposition \ref{prop:3nn0} depending on the form $ij\ci kl\cd C$ or $ij\ci ik\cd C$.

Singleton-transitivity also follows from the above argument.

2.\ The proof follows from the first part and the duality (Lemma \ref{lem:dual}).
%
%
%
%
%
\end{proof}
\begin{theorem}
Let a distribution $P$ be defined over an exchangeable random network, and $\sk(P)$ be empty. Suppose also that one of the following cases holds for $i,j,k,l$ that are all different:
\begin{itemize}
  \item[a)] $ij\ci kl$ and $ij\ci ik$;
  \item[b)] $ij\ci kl\cd V\setminus\{ij,kl\}$ and $ij\ci ik\cd V\setminus\{ij,ik\}$.
\end{itemize}
Then $P$ is faithful to the empty graph if and only if $P$ satisfies the intersection and composition properties.
\end{theorem}
\begin{proof}
The proof follows from Propositions \ref{prop:3nnnn0} and \ref{prop:und}.
\end{proof}
\section{Summary and discussion}\label{sec:5}
Our results concern the conditional independence models induced by exchangeable distributions for random vectors and random networks. We have shown that exchangeable random vectors are completely independent of each other or completely dependent of each other if they satisfy intersection and composition properties. In addition, they are marginally independent if there exists at least one independence statement and the intersection property is satisfied. The intersection property is well-understood, and we know that a positive joint density is a sufficient condition for it to hold; thus it is particularly important to study the composition property for exchangeable random vectors, which, in this case, is simplified to $A\ci B\cd C\Rightarrow A\ci (B\cup D)\cd C$, for every disjoint $D$.

For exchangeable random networks, as it turned out, the situation is much more complicated. As an important extension of our results in \cite{lau17}, we showed that the independence structures of exchangeable random networks that can be represented by a graph in graphical model sense are one of the six possible cases: completely dyadic-independent, faithful to the undirected or bidirected incidence graphs, faithful to the undirected or bidirected complement of the incidence graph, or completely dyadic-dependent.

The undirected and bidirected versions of the incidence graph and its complement are in fact dual to each other, so in a sense there are four regimes available. In other words, with exchangeability and duality factored in, one passes from the skeleton to the graphical Markov equivalence class. Exploiting this duality, all the results for the undirected case can be extended to the bidirected case. In addition, the remaining four cases are just two cases modulo graph complement as well. Although we failed to do so, it would
be especially nice if this duality could be understood, so that one can simply present the results with two-fold duality arguments.

We have provided a simple test to decide in which of the six regimes an exchangeable random network lies in cases when it has a ``structured" independence structure. The main two elements of the four ``non-trivial" cases is whether an independence is of form $ij\ci kl\cd C$ or $ij\ci ik\cd C$; and whether $C_{ij,kl}=\{ik,il,jk,jl\}$ is in $C$ or is disjoint from $C$.

We, in fact, do not have ``necessary" and sufficient conditions for whether an exchangeable random network is structured, but rather sufficient conditions that, in addition to the expected intersection and composition properties, are mainly based on whether a minimal (maximal) separator is invariant under the mentioned node swaps of the network. For testing purposes, it is important to stress that the only conditioning set that needs to be tested are the minimal ones, which would significantly improve the computational complexity of any relevant algorithms.

Indeed, in practice, it is more important to understand the independence structures of exchangeable statistical network models for  random networks with (in most situations) binary dyads. One point is that binary distributions always satisfy singleton-transitivity \cite{drts09}, a necessary condition for faithfulness,  although under our sufficient assumptions this condition is automatically satisfied. In general, however, it would be useful to study which actual exchangeable models for networks (such as exchangeable exponential random graph models \cite{was96}) satisfy the provided conditions, both when we deal with binary random networks or weighted ones.

\appendix

\section{Proof of Theorem \ref{thm:2}}\label{app}
As mentioned in Section \ref{sec:41}, we focus on the class of chain mixed graphs, which contains simultaneous undirected, directed, or bidirected edges, with the separation criterion introduced in \cite{sadl16}. We refrain from defining this class explicitly as it is not needed for our purpose. However, we note that we can focus only on simple graphs since it was shown in \cite{sadl16} that for any (non-simple) chain mixed graph there is a Markov equivalent simple graph (the collection of which constitutes the class of \emph{anterial graphs}). We also only focus on \emph{maximal graphs}, which are graphs where a missing edge between vertices $u$ and $v$ implies that there exists a separation statement of form $u\dse v\cd C$, for some $C$ -- again it was shown in \cite{sadl16} that for any (non-maximal) chain mixed graph there is a Markov equivalent maximal graph.

First, we need the following additional definitions: A \emph{section} $\rho$ of a walk  is a maximal subwalk consisting only of lines, meaning that there is no other subwalk that only consists of lines and includes $\rho$. Thus, any walk decomposes uniquely into sections; these are not necessarily edge-disjoint and sections may also be single vertices. A section $\rho$ on a walk $\omega$ is called a  \emph{collider section} if one of the following walks is a subwalk of $\omega$: $i\fra\rho\fla\,j$, $i\arc\rho\fla\,j$, $i\arc\rho\arc\,j$. All other sections on $\omega$ are called \emph{non-collider} sections. A \emph{trisection} is a walk $\langle i,\rho,j\rangle$, where $\rho$ is a section.  If in the trisections, $i$ and $j$ are distinct and not adjacent then the trisection is called \emph{unshielded}. We say that a trisection is collider or non-collider if its section $\rho$ is collider or non-collider respectively.

We say that a walk $\omega$ in a graph is \emph{connecting} given $C$ if all collider sections of $\omega$ intersect $C$ and all non-collider sections are disjoint from $C$. For pairwise disjoint subsets $A,B,C$, we say that $A$ and $B$ are separated by $C$ if there are no connecting walks between $A$ and $B$ given $C$, and we use the notation $A\dse B \cd C$.
%
\begin{lemma}\label{lem:3n}
If two maximal graphs $G$ and $H$ are Markov equivalent then $G$ and $H$ have the same unshielded collider trisections.
\end{lemma}
\begin{proof}
Because of maximality, $G$ and $H$ have the same skeleton. An unshielded trisection in these graphs cannot be a collider in one and a non-collider in the other. This is because if that is the case (say an unshielded  trisection between $i$ and $j$ and separation $i\dse j\cd C$), by Markov equivalence, it implies that an inner vertex  of the corresponding section is not in $C$ in one, but in $C$ in the other, which is a contradiction.
\end{proof}
We will not need the converse of the above lemma, but only a weaker result:
\begin{lemma}\label{lem:3nn}
If there are no unshielded collider trisections in $G$ then $G$ is Markov equivalent to $\sk(G)$.
\end{lemma}
\begin{proof}
First, we show that if $A\notdse B\cd C$ in $G$ then $A\notdse B\cd C$ in $\sk(G)$: It holds that there is a connecting walk $\omega$ in $G$ between $A$ and $B$ given $C$. If there are no collider sections on $\omega$ then no vertex on $\omega$ is in $C$. Therefore, $\omega$ is a connecting walk in $\sk(G)$. If there is a collider section with endpoints $i$ and $j$ on $\omega$ then it has to be shielded. We can now replace this collider section with the $ij$ edge. Applying this method repeatedly, we obtain a walk that has no vertex in $C$, and is, therefore, connecting in $\sk(G)$.

Now, we show that if $A\dse B\cd C$ in $G$ then $A\dse B\cd C$ in $\sk(G)$: Consider an arbitrary path between $A$ and $B$ in $\sk(G)$. We need to show that there is a vertex on this path that is in $C$. Consider this path in $G$ and call it $\omega$. If all sections on $\omega$ are non-collider then there must be a vertex on $\omega$ that is in $C$, and we are done. Hence, consider a collider section with endpoints $i$ and $j$ on $\omega$. This section is shielded; thus, replace the section with the $ij$ edge. By repeating this procedure, we either obtain a path, with a subset of vertices of $\omega$, whose sections are all non-collider; or we eventually obtain an edge between the endpoints of $\omega$, which is impossible.
\end{proof}
The following lemma extends the concept of exchangeability for random networks to graphs in graphical models:
\begin{lemma}\label{lem:3}
Suppose that a distribution $P$  over an exchangeable random network with node set $\mathcal{N}$ is faithful to a graph $G$, and let $\pi$ be a permutation function on $\mathcal{N}$. Let also $H$ be the graph obtained by permuting the vertices of $G$ by vertex $ij$ being mapped to $\pi(i)\pi(j)$. Then $P$ is faithful to $H$; hence, $G$ and $H$ are Markov equivalent.
\end{lemma}
\begin{proof}
Let $\mathcal{J}_\pi(P)$ be the independence model obtained from $\mathcal{J}(P)$ by mapping independence statements $A\ci B\cd C$ to $\pi(A)\ci \pi(B)\cd \pi(C)$, where $\pi(A)=\{\pi(i)\pi(j):\n ij\in A\}$, etc. It is obvious that $\mathcal{J}_\pi(P)$ is faithful to $H$. Because of exchangeability, we also have that $\mathcal{J}(P)=\mathcal{J}_\pi(P)$. Therefore, $P$ is faithful to $H$. Hence, $G$ and $H$ are Markov equivalent.
%
\end{proof}
We are now ready to provide the proof of Theorem \ref{thm:2}:
\begin{proof}[Proof of Theorem \ref{thm:2}]
Cases 1 and 6 are trivial. Thus, we need to consider cases 2 and 3 of Proposition \ref{prop:7}. By Lemma \ref{lem:3nn}, if there are no unshielded collider trisections in $G$ then $G$ is Markov equivalent to $L_{-}(n)$ or $L^c_{-}(n)$, respectively. Thus, suppose that there is an unshielded collider trisection in $G$.

For case  2 of Proposition \ref{prop:7}, we can assume that there is an edge $12,13$ in an unshielded collider trisection, such that there is an arrowhead at $13$ on $12,13$. Now, consider an arbitrary edge $ij,ik$ in $G$. Let $\pi$ be a permutation that only swaps $i$ and $1$, $j$ and $2$, and $k$ and $3$, and call the resulted graph $H$. (Notice that it is possible that $j=3$ and $k=2$.) By Lemma \ref{lem:3}, $H$ and $G$ are Markov equivalent. In addition, $ik$ is in an unshielded collider trisection in $H$ with an arrowhead at vertex $ik$ on $ij,ik$. Hence, by Lemma \ref{lem:3n}, $ik$ is in an unshielded collider trisection in $G$ with an arrowhead at vertex $ik$ on $ij,ik$. Since $i,j,k$ are arbitrary, and in particular, every edge could be mapped to $ij,ik$ by a permutation, we conclude that there is an arrowhead at every vertex on every edge in $G$. Therefore, $G$ is a bidirected graph (and Markov equivalent to $L_{\leftrightarrow}(n)$).

%

For case  3 of Proposition \ref{prop:7}, we assume that there is an edge $12,34$ in an unshielded collider trisection, such that there is an arrowhead at vertex $34$ on $12,34$. In this case, we apply a similar method to the previous case, but by a permutation that only swaps $i$ and $1$, $j$ and $2$, $k$ and $3$, and $l$ and $4$.
\end{proof}
\section*{Acknowledgements}
The author is grateful to Steffen Lauritzen and Alessandro Rinaldo for raising this problem in one of our numerous conversations. The author is also truly thankful to the two anonymous referees, whose comments substantially improved the paper.
\bibliographystyle{imsart-nameyear}
\bibliography{bib}

\end{document}